\documentclass[12pt,a4paper]{amsart}

\usepackage{amssymb}
\usepackage{amsthm}
\usepackage{amsmath}
\usepackage{graphicx}
\usepackage{enumerate}
\usepackage{xcolor}
\usepackage{tikz}
\usepackage{graphbox}
\usepackage{orcidlink}
\usepackage{mathrsfs}  
\usepackage{hyperref}
\hypersetup{colorlinks=true,allcolors=blue}
\usepackage[margin=2.5cm]{geometry}

\newtheorem{theorem}{Theorem}[section]
\newtheorem{proposition}[theorem]{Proposition}

\newtheorem{lemma}[theorem]{Lemma}

\theoremstyle{definition}
\newtheorem{example}[theorem]{Example}

\newcommand{\II}{\mathbb{I}}
\newcommand{\RR}{\mathbb{R}}

\newcommand{\CC}{\mathcal{C}}
\newcommand{\cQ}{\mathcal{Q}}

\title[The exact region determined by $\phi$, $\gamma$ and $\tau$]{The exact region determined by Spearman's footrule, Gini's gamma and Kendall's tau}

\author[D. Kokol Bukovšek]{Damjana Kokol Bukovšek \orcidlink{0000-0002-0098-6784} }
\address{University of Ljubljana, School of Economics and Business, and Institute of Mathematics, Physics and Mechanics, Ljubljana, Slovenia}
\email{Damjana.Kokol.Bukovsek@ef.uni-lj.si}

\author[P. Lazić]{Petra Lazić \orcidlink{0000-0003-3064-773X} } 
\address{University of Ljubljana, Faculty of Mathematics and Physics, Ljubljana, Slovenia, and University of Zagreb, Faculty of Science, Department of Mathematics, Zagreb, Croatia}
\email{Petra.Lazic@fmf.uni-lj.si}

\author[B. Mojškerc]{Blaž Mojškerc \orcidlink{0000-0001-8096-355X} }
\address{University of Ljubljana, School of Economics and Business, and Institute of Mathematics, Physics and Mechanics, Ljubljana, Slovenia}
\email{Blaz.Mojskerc@ef.uni-lj.si}

\author[N. Stopar]{Nik Stopar \orcidlink{0000-0002-0004-4957} } 
\address{University of Ljubljana, Faculty of Civil and Geodetic Engineering, University of Ljubljana, Faculty of Mathematics and Physics, and Institute of Mathematics, Physics and Mechanics, Ljubljana, Slovenia}
\email{Nik.Stopar@fgg.uni-lj.si}


\begin{document}

\begin{abstract}
Concordance measures are used to express the degree of association between random variables. Practitioners may use several distinct concordance measures to narrow the space of possible dependence structures. Consequently, the relations between different (weak) concordance measures have been extensively studied in recent years. The goal of this paper is to study the relation between Kendall's tau, Gini's gamma and Spearman's footrule. In particular,  we describe the exact region determined by these three measures, using shuffles of $M$ and ordinal sums of copulas. We also provide the formulas for five main (weak) concordance measures  and Chatterjee's xi of ordinal sums of copulas.
\end{abstract}

\maketitle

\section{Introduction}

Various concordance measures can be used to describe the association between random variables.  Given two random variables $X$ and $Y$, the concordance measure between them depends only on the copula associated with their joint distribution. Traditionally, Spearman's rho and Kendall's tau are the most commonly used in practice and there is an abundance of literature discussing their properties. Later, several other (weak) concordance measures were introduced, such as Gini's gamma, Blomqvist's beta and Spearman's footrule. The central interest of this paper are the relations between Kendall's tau, Gini's gamma and Spearman's footrule, therefore, we provide some references for these three. 
Kendall's tau was investigated in \cite{FuSch2,JaMa,KGJT,Wys},
Gini's gamma in \cite{CoNi,GeNeGh,Nels1998},
and Spearman's footrule in \cite{DiaGra,GeNeGh,UbFl,SSQ}.

Each concordance measure captures a specific aspect of association, so practitioners tend to use several of them to narrow the space of possible dependence structures. Hence, the relations between different (weak) concordance measures of copulas are of interest and have been extensively studied.

The problem of the exact region determined by Spearman's rho and Kendall's tau was open since 1960's and was solved in \cite{ScPaTr}. The relation between Blomqvist's beta and other (weak) concordance measures is relatively easy to describe (see \cite{KoBuKoMoOm3}), while relations between others are more challenging. The region determined by Gini's gamma and Spearman's footrule is characterized in \cite{KoBuMo}. The problem of the region determined by Spearman's rho and Spearman's footrule is partially solved in \cite{KoBuSt2}, see also \cite{TsScTr}. The lower bound for this region is determined exactly, while for the upper bound only a good approximation is known. The relations between Kendall's tau and other (weak) concordance measures are considered in \cite{KoBuSt}. The region determined by Spearman's rho  and Gini's gamma is still open.

In 2021, Chatterjee's coefficient of rank correlation (Chatterjee's xi) was introduced in \cite{Chat} and became extremely popular. It quantifies the strength of functional dependence between random variables $X$ and $Y$. Even though this is not a concordance measure, it has still been compared to common measures of concordance, see \cite{AnRo25, FuLiSc26, Rock25}.

The paper \cite{FuLiSc26} emphasizes the importance of studying higher-dimensional versions of the problem, namely, capturing the connections between more than two concordance measures simultaneously. The first treatment of a triplet of concordance measures, Blomqvist's beta, Spearman's footrule and Gini's gamma, was presented in \cite{KoBuMo26}. The goal of the present paper is to study the relation between Kendall's tau, Spearman's footrule and Gini's gamma. We first prove the lower and upper bound for Kendall's tau of a copula if its Spearman's footrule and Gini's gamma are given. Then we give several examples to demonstrate that these bounds are attained by shuffles of $M$. It turns out that the set of all triplets $(\phi(C), \gamma(C), \tau(C))$ for all copulas $C$ is a polyhedron.

To prove our results, we use ordinal sums of copulas, hence, we need the formulas expressing concordance measures of an ordinal sum of copulas in terms of concordance measures of its summands. We derive these formulas for all five most common (weak) concordance measures and Chatterjee's xi. These results may be of independent interest to the reader. Some of them already appeared in the literature without proofs.

The paper is structured as follows. In Section~\ref{sec:prelim} we introduce the notation, give the basic definitions and recall some results that we need later. In Section~\ref{sec:ordinal} we prove the formulas for concordance measures and Chatterjee's xi of an ordinal sum of copulas. In Section~\ref{sec:region} we present our main result, i.e., we describe the exact region determined by Spearman's footrule, Gini's gamma and Kendall's tau. Moreover, we provide examples of copulas attaining the boundary of the obtained region. We end the paper with some concluding remarks, including a discussion on the possible range of Kendall's tau when  Spearman's footrule and Gini's gamma of a copula are given.

\section{Preliminaries} \label{sec:prelim}

Let $\II$ be the unit interval $[0,1]\subseteq\RR$ and $u_1, u_2, v_1, v_2 \in \II$ be such that $u_1 \le u_2$ and $v_1 \le v_2$. The Cartesian product of intervals $R=[u_1,u_2]\times [v_1,v_2]$ is called a {\it rectangle} in $\II^2$. Let $H:\II^2 \to \RR$ be a real function. We define the $H$-volume of $R$ as
$V_H(R)=H(u_2,v_2)-H(u_2,v_1)-H(u_1,v_2)+H(u_1,v_1)$.
A {\it bivariate copula} is a function $C:\II^2 \to \II$ with the following properties:
\begin{itemize}
    \item $C(0,v)=C(u,0)=0$ for all $u,v \in \II$ ($C$ is {\it grounded}),
    \item $C(u,1)=u$ and $C(1,v)=v$ for all $u,v \in \II$ ($C$ has {\it uniform marginals}), and
    \item $V_C(R) \geq 0$ for every rectangle $R \subseteq \II^2$ ($C$ is {\it $2-$increasing}).
\end{itemize}
Bivariate copulas are functions of two variables which couple bivariate distribution functions with their one-dimensional marginal distribution functions, a famous Theorem by Sklar \cite{Sklar}.
The set of all bivariate copulas will be denoted by $\CC$. It is well known that this set is compact in the sup norm. For any copula $C$ its diagonal will be denoted by $\delta_C$ and its opposite diagonal by $\omega_C$, i.e., $\delta_C(u)=C(u,u)$ and $\omega_C(u)=C(u,1-u)$ for all $u \in \II$.

Given two copulas $C$ and $D$, we denote $C\le D$ if $C(u,v)\leqslant D(u,v)$ for all $(u,v)\in\II^2$. This is the so-called \textit{pointwise order} of copulas.
For any copula $C$ we have $W \le C \le M$, where $W(u,v)=\max\{0,u+v-1\}$ and $M(u,v)=\min\{u,v\}$ are the lower and upper \textit{Fr\'echet-Hoeffding bounds} for the set of all copulas.
Furthermore, any copula $C$ induces reflected copulas $C^{\sigma_1}$ and $C^{\sigma_2}$ defined by $C^{\sigma_1}(u,v) =v-C(1-u,v)$ and $C^{\sigma_2}(u,v)=u-C(u,1-v)$ for all $(u,v) \in \II^2$. Note that $\delta_{C^{\sigma_2}}(u)=u-\omega_C(u)$. 

Let $h \colon \II \to \II$ be a measure preserving bijection, where $\II$ is equipped with the Lebesgue measure $\lambda$. Then the function defined by
\begin{equation}\label{eq:measure_preserving}
C(u,v)=\lambda\big(\{t \in \II \mid t\le u, h(t)\le v \}\big)
\end{equation}
is a copula whose mass is concentrated on the graph of $h$, i.e., $\mu_C\big(\{(t,h(t)) \mid t \in \II\}\big)=1$.
Particular examples of such copulas are the so-called
\emph{shuffles of $M$}. A shuffle of $M$
$$ C = M(n, J, \pi, \omega)$$
is determined by a positive integer $n$, a partition
$J=\{J_1,J_2,\ldots, J_n\}$ of the interval $\II$ into $n$ pieces, where $J_i=[u_{i-1},u_i]$ and $0 =u_0 \le u_1 \le u_2 \le \ldots \le u_{n-1} \le u_n=1$, shortly written as $(n-1)$-tuple of splitting points $J=(u_1, u_2, \ldots, u_{n-1})$,
a permutation $\pi \in S_n$, written as $n$-tuple of images $\pi = (\pi(1), \pi(2), \ldots, \pi(n))$, and a mapping $\omega: \{1, 2, \ldots, n\} \to \{-1, 1\}$, written as $n$-tuple of images $\omega = (\omega(1), \omega(2), \ldots, \omega(n))$. The mass of $C$ is concentrated on the main diagonals of the squares $[u_{i-1},u_i] \times  [v_{\pi(i)-1} \times v_{\pi(i)}]$ where $0= v_0 \le v_1 \le v_2\le \ldots \le v_{n-1} \le v_n = 1$. 
Hence, $C$ is defined by the measure preserving bijection $h_C \colon \II \to \II$, given by
$$ h_C(u) =  \begin{cases}
            u + v_{\pi(i)-1} - u_{i-1}; &  u \in (u_{i-1}, u_i),\ \omega(i) = 1, \\
            v_{\pi(i)} +  u_{i-1} - u; &  u \in (u_{i-1}, u_i),\ \omega(i) = -1, \\
            v_i; & u= u_i.
            \end{cases}
$$
Furthermore, $C$ is the copula of uniformly distributed random variables $U$ and $V$ on $\II$ with the property $P(V=h_C(U)) =1$. For more details see \cite[\S3.2.3]{Nels}.

In \cite{Scar} Scarsini introduced formal axioms for \emph{concordance measures}.
These are mappings that assign to each copula a real number in $[-1,1]$ and are meant to measure the degree of concordance/discordance between the components of random vectors.
Recall that two observations $(x,y)$ and $(x',y')$ from a random vector $(X,Y)$ are concordant if $(x-x')(y-y')>0$ and discordant if $(x-x')(y-y')<0$.
For the formal definition of concordance measures and further details we refer the reader to \cite{DuSe} and \cite{Nels}. Here we give the properties of concordance measures, which we will need in the sequel: if $\kappa$ is a concordance measure, then $\kappa(M) = 1$, $\kappa(C^{\sigma_1}) = \kappa(C^{\sigma_2}) = -\kappa(C)$ for any copula $C \in \CC$, $\kappa$ is continuous with respect to the pointwise convergence, and $\kappa$ is monotone increasing with respect to the pointwise order.

Many of the most important bivariate concordance measures, including Kendall's tau and Gini's gamma, can be expressed with the so-called \emph{concordance function} $\cQ$, introduced by Kruskal \cite{Krus}.
If $(X_1,Y_1)$ and $(X_2,Y_2)$ are pairs of continuous random variables,
then the concordance function of random vectors $(X_1,Y_1)$ and $(X_2,Y_2)$ depends only on the corresponding copulas $C_1$ and $C_2$ and is given by (see \cite[Theorem 5.1.1]{Nels})
\begin{equation}\label{concordance} \cQ=\cQ(C_1,C_2)= 4 \int_{\II ^2} C_2(u,v) \, dC_1(u,v) -1.
\end{equation}
It turns out that the concordance function is symmetric in its arguments, i.e., $\cQ(C_1,C_2)=\cQ(C_2,C_1)$, and has several other useful properties, see \cite[Corollary 5.1.2]{Nels} and \cite[\S{3}]{KoBuKoMoOm2}.

The most important concordance measures include Spearman's rho, Kendall's tau, Gini's gamma and Blomqvist's beta. 
With the concordance function at hand, \textit{Spearman's rho} can be defined by
\begin{equation}\label{rho}
\rho(C)=3\cQ(C,\Pi)= 12\int_{\II ^2} C(u,v) \, du \, dv - 3,
\end{equation}
\textit{Kendall's tau} by
\begin{equation}\label{tau}
\tau(C)=\cQ(C,C)= 4\int_{\II ^2} C(u,v) \, dC(u,v) - 1,
\end{equation}
and \textit{Gini's gamma} by
\begin{equation}\label{gamma}
\gamma(C)=\cQ(C,M)+\cQ(C,W) = 4\int_0^1 \delta_C(u) \, du + 4\int_0^1 \omega_C(u) \, du - 2.
\end{equation}
Blomqvist's beta is defined as 
\begin{equation}\label{beta}
\beta(C)=4C(\tfrac12,\tfrac12)- 1.
\end{equation}
If $h_C \colon \II \to \II$ is a measure preserving bijection and $C$ a copula whose mass is concentrated on the graph of $h_C$, then $\tau(C)$ can be expressed as 
\begin{equation}\label{eq:tau-h}
\tau(C)=4\int_0^1 C(u,h_C(u)) \, du - 1.
\end{equation}

In \cite{Lieb} Liebscher considered \textit{weak concordance measures}, which are slight\-ly more general mappings than concordance measures (the formal definition can be found in Liebscher's paper).
The most important example of a weak concordance measure is \textit{Spearman's footrule} defined by
\begin{equation}\label{phi}
\phi(C)=\tfrac12\left(3\cQ(C,M)-1\right)=6\int_0^1 \delta_C(u) \, du - 2.
\end{equation}
Spearman's footrule is not a true concordance measure since its range is $[-\frac12,1]$.

Finally, we recall the definition of Chatterjee's xi, which is not a concordance measure, but rather a nonsymmetric coefficient of correlation.
Every copula $C\in\CC$ is increasing in each variable, hence partially differentiable almost everywhere. We denote by $\partial_1 C(u,v)$ the partial derivative of $C$ with respect to the first variable. One of the equivalent definitions of Chatterjee's coefficient of rank correlation is the following (see \cite{DeSiSt13, Rock25})
\begin{equation} \label{eq:xi}
    \xi(C)= 6 \int_{\II^2} (\partial_1 C(u,v))^2 \, du \, dv -2.
\end{equation}

In the next proposition we recall the relations between Kendall's tau, Gini's gamma and Spearman's footrule, which will be needed in the sequel.

\begin{proposition}[\cite{KoBuMo,KoBuSt}]\label{pro:old-results}
For any copula $C\in\CC$ we have
\begin{align}
    \tfrac43 \phi(C) - \tfrac13 &\leq \gamma(C) \leq \min\big\{\tfrac43 \phi(C)+\tfrac16, \tfrac23 \phi(C)+\tfrac13\big\}, \label{eq:gamma-phi}\\
    \tfrac43\phi(C)-\tfrac13 &\le \tau(C) \le \tfrac23\phi(C)+\tfrac13, \label{eq:tau-phi}\\
    \max\big\{\tfrac23 \gamma(C)-\tfrac13, 2\gamma(C)-1\big\} &\le \tau(C) \le \min\big\{\tfrac23 \gamma(C)+\tfrac13,2\gamma(C)+1\big\}.\label{eq:gamma-tau}
\end{align}
The bounds are attained by shuffles of $M$.
\end{proposition}

\section{Concordance measures of an ordinal sum} \label{sec:ordinal}

In this section we give the formulas for five (weak) concordance measures (Spearman's rho, Kendall's tau, Spearman's footrule, Gini's gamma, and Blomqvist's beta), and Chatterjee's xi of an ordinal sum of copulas.

Let $\{(a_k, b_k) \mid k=1,2, \ldots, n\}$ be a finite family of disjoint open subintervals of $\II$ and $\{B_k \mid k=1,2, \ldots, n\}$ a family of copulas. Then the \emph{ordinal sum} (sometimes also called \emph{$M$-ordinal sum}) $B$ of $\{B_k \mid k=1,2, \ldots, n\}$ with respect to $\{(a_k, b_k) \mid k=1,2, \ldots, n\}$ is a copula defined by 
\begin{equation} \label{eq:ordinal_sum}
B(u, v) = \begin{cases}
		a_k + (b_k-a_k)B_k\big(\frac{u-a_k}{b_k-a_k}, \frac{v-a_k}{b_k-a_k}\big); & (u, v) \in [a_k, b_k]^2, k=1, 2, \dots, n,\\
			\min\{u,v\}; & \text{otherwise},
    \end{cases}
\end{equation}
(see \cite{DuKlSaSe}, compare also \cite[\S3.2.2]{Nels}).

We start with the formula for Spearman's rho. It already appeared in \cite{KoBuSt2}, but without a proof, so we give the proof here. 

\begin{proposition}\label{prop:ordinal-rho}
   Spearman's rho of an ordinal sum $B$, given in \eqref{eq:ordinal_sum}, is
    \begin{equation} \label{eq:ordinal-rho}
        \rho(B) =  1 - \sum_{k=1}^n (b_k-a_k)^3(1-\rho(B_k)). 
    \end{equation}  
\end{proposition}

\begin{proof} Since $B(u,v) = M(u,v)$ outside the union of squares $[a_k, b_k]^2$, we have
\begin{align*}
\int_{\II^2} B(u,v) \, du\, dv &= \int_{\II^2} M(u,v) \, du\, dv + \int_{\II^2} (B(u,v)-M(u,v)) \, du \, dv\\
    &= \tfrac13 + \sum_{k=1}^n  \int_{[a_k, b_k]^2} (B(u,v)-M(u,v)) \, du\, dv \\
    &= \tfrac13 + \sum_{k=1}^n  \int_{[a_k, b_k]^2} B(u,v) \, du\, dv - \sum_{k=1}^n  \int_{[a_k, b_k]^2} M(u,v) \, du \, dv \\
    &= \tfrac13 + \sum_{k=1}^n  \int_{[a_k, b_k]^2} B(u,v) \, du\, dv - \sum_{k=1}^n (b_k-a_k)^2 (\tfrac23 a_k+ \tfrac13 b_k).
\end{align*}
Since
\begin{align*}
\int_{[a_k, b_k]^2} B(u, v) \, du \, dv &= \int_{[a_k, b_k]^2} \big(a_k + (b_k-a_k)B_k(\tfrac{u-a_k}{b_k-a_k},\tfrac{v-a_k}{b_k-a_k})\big) \, du \, dv \\ 
&= a_k(b_k-a_k)^2 + (b_k-a_k)^3\int_{\II^2}B_k(t,s) \, dt \, ds \\ 
&= a_k(b_k-a_k)^2 + (b_k-a_k)^3\cdot\frac{\rho(B_k)+3}{12},
\end{align*}
we obtain
\begin{align*}
\rho(B) &= 12\int_{\II^2} B(u,v) \, du \, dv -3\\  
&= 1 + 12\sum_{k=1}^n  \int_{[a_k, b_k]^2} B(u,v) \, du \, dv - 12\sum_{k=1}^n (b_k-a_k)^2 (\tfrac23 a_k+ \tfrac13 b_k) \\ 
&= 1 + \sum_{k=1}^n \left(12a_k(b_k-a_k)^2 -(b_k-a_k)^2 (8 a_k+ 4 b_k) + (b_k-a_k)^3(\rho(B_k)+3) \right) \\ 
&= 1 + \sum_{k=1}^n \left(-4(b_k-a_k)^3 + (b_k-a_k)^3(\rho(B_k)+3) \right) \\ 
&= 1 - \sum_{k=1}^n  (b_k-a_k)^3(1-\rho(B_k)),   
\end{align*}
which finishes the proof.
\end{proof}

We continue with the formula for Kendall's tau of an ordinal sum. 

\begin{proposition}\label{prop:ordinal-tau}
   Kendall's tau of an ordinal sum $B$, given in \eqref{eq:ordinal_sum}, is
    \begin{equation} \label{eq:ordinal-tau}
        \tau(B) =  1 - \sum_{k=1}^n (b_k-a_k)^2(1-\tau(B_k)). 
    \end{equation}  
\end{proposition}

\begin{proof}
Since $B(u,v) = M(u,v)$ outside the union of squares $[a_k, b_k]^2$, we calculate
\begin{align*}
\tau(B) &= - 1 + 4\int_{\II^2} B(u, v) \, dB(u,v) \\
&= -1 + 4 \sum_{k=1}^n  \int_{[a_k, b_k]^2} B(u, v) \, dB(u,v)  + 4 \int_{\II\setminus\cup_{k=1}^n(a_k,b_k)} u \, du\\
&= -1 + 4 \sum_{k=1}^n  \int_{[a_k, b_k]^2} B(u, v) \, dB(u,v) + 4 \int_0^1 u \, du - 4 \sum_{k=1}^n \int_{a_k}^{b_k} u \, du \\
&= 1 + 4 \sum_{k=1}^n  \int_{[a_k, b_k]^2} B(u, v) \, dB(u,v) - 2 \sum_{k=1}^n (b_k^2-a_k^2).
\end{align*}
Now,
\begin{align*}
\int_{[a_k, b_k]^2} B(u, v) \, dB(u,v) &= \int_{[a_k, b_k]^2} \big(a_k + (b_k-a_k)B_k(\tfrac{u-a_k}{b_k-a_k},\tfrac{v-a_k}{b_k-a_k})\big) \, dB(u,v) \\ 
&= \int_{[a_k, b_k]^2} a_k \, dB(u,v) + (b_k-a_k)\int_{[a_k, b_k]^2}B_k\big(\tfrac{u-a_k}{b_k-a_k},\tfrac{v-a_k}{b_k-a_k}\big) \, dB(u,v) \\ 
&= \int_{[a_k, b_k]^2} a_k(b_k-a_k) \, dB_k(\tfrac{u-a_k}{b_k-a_k},\tfrac{v-a_k}{b_k-a_k}) \\
&\qquad + (b_k-a_k)^2\int_{[a_k, b_k]^2}B_k(\tfrac{u-a_k}{b_k-a_k},\tfrac{v-a_k}{b_k-a_k}) \, dB_k(\tfrac{u-a_k}{b_k-a_k},\tfrac{v-a_k}{b_k-a_k}) \\ 
&= a_k(b_k-a_k) \int_{\II^2} dB_k(t,s) + (b_k-a_k)^2\int_{\II^2}B_k(t,s) \, dB_k(t,s) \\ 
&= a_k(b_k-a_k) + (b_k-a_k)^2 \cdot \frac{\tau(B_k)+1}{4}.
\end{align*}
We obtain
\begin{align*}
\tau(B) &= 1 + 4 \sum_{k=1}^n  a_k(b_k-a_k) + 4 \sum_{k=1}^n  (b_k-a_k)^2\cdot\frac{\tau(B_k)+1}{4} - 2 \sum_{k=1}^n (b_k^2-a_k^2) \\  
&= 1 -2 \sum_{k=1}^n  (b_k-a_k)^2 + \sum_{k=1}^n  (b_k-a_k)^2(\tau(B_k)+1)  \\    
&= 1 - \sum_{k=1}^n  (b_k-a_k)^2(1-\tau(B_k)),   
\end{align*}
which finishes the proof.
\end{proof}

A version of the formula for Spearman's footrule of an ordinal sum appeared already in \cite{TsScTr}, under the assumption that the union of the family of closed subintervals $\{[a_k, b_k] \mid k=1,2, \ldots, n\}$ equals the whole interval $\II$. Here we prove it for the general case. 

\begin{proposition}\label{prop:ordinal-phi}
   Spearman's footrule of an ordinal sum $B$, given in \eqref{eq:ordinal_sum}, is
    \begin{equation} \label{eq:ordinal-phi}
        \phi(B) =  1 - \sum_{k=1}^n (b_k-a_k)^2(1-\phi(B_k)). 
    \end{equation}  
\end{proposition}

\begin{proof}
Note that the diagonal section $\delta_B$ of copula $B$ equals
\begin{align*}
\delta_B(u) &= \begin{cases}
		a_k + (b_k-a_k)\delta_{B_k}(\frac{u-a_k}{b_k-a_k}); & u \in [a_k, b_k], k=1, 2, \dots, n,\\
		u; & \text{otherwise} 
    \end{cases} \\
    &= u + \begin{cases}
		a_k - u + (b_k-a_k)\delta_{B_k}(\frac{u-a_k}{b_k-a_k}); & u \in [a_k, b_k], k=1, 2, \dots, n,\\
		0; & \text{otherwise},
    \end{cases}
\end{align*}
so that 
\begin{align*}
\int_0^1 \delta_B(u) \, du &= \int_0^1 u \, du + \sum_{k=1}^n\int_{a_k}^{b_k}
		\big(a_k - u + (b_k-a_k)\delta_{B_k}(\tfrac{u-a_k}{b_k-a_k})\big) \, du\\
    &= \tfrac12 + \sum_{k=1}^n\left(-\tfrac12(b_k-a_k)^2 + (b_k-a_k)^2\int_0^1 \delta_{B_k}(t) \, dt\right) \\
    &= \tfrac12 + \sum_{k=1}^n\left(-\tfrac12(b_k-a_k)^2 + (b_k-a_k)^2 \cdot\frac{\phi(B_k)+2}{6}\right) \\    
    &= \tfrac12 + \sum_{k=1}^n(b_k-a_k)^2\cdot\frac{\phi(B_k)-1}{6} ,
\end{align*}
which implies
$$\phi(B) = 3 + \sum_{k=1}^n(b_k-a_k)^2(\phi(B_k)-1) - 2 = 1 - \sum_{k=1}^n (b_k-a_k)^2(1-\phi(B_k))$$
and finishes the proof.
\end{proof}

Note that in the formulas for Spearman's rho, Kendall's tau, and Spearman's footrule of an ordinal sum the position of the subintervals $(a_k, b_k)$ is not important, only their size $b_k - a_k$ matters. For Gini's gamma this is no longer true. In order to give the result for Gini's gamma, we first prove a simple lemma. 

\begin{lemma} \label{lem:gamma-phi}
For any copula $C \in \CC$ we have 
$$\gamma(C) = \tfrac23 (\phi(C)-\phi(C^{\sigma_2})).$$
\end{lemma}

\begin{proof}
For any copula $C$ we have
$$\int_0^1 \delta_{C^{\sigma_2}}(u) \, du=\int_0^1 C^{\sigma_2}(u,u) \, du=\frac12-\int_0^1 C(u,1-u) \, du= \frac12-\int_0^1 \omega_C(u) \, du,$$
hence
\begin{align*}
\gamma(C) &=4\int_0^1 \delta_{C}(u) \, du+4\left(\frac12-\int_0^1 \delta_{C^{\sigma_2}}(u) \, du\right)-2 \\
&=4\cdot\frac{\phi(C)+2}{6}-4\cdot\frac{\phi(C^{\sigma_2})+2}{6} \\
&=\tfrac23 (\phi(C)-\phi(C^{\sigma_2})),
\end{align*}    
which finishes the proof.
\end{proof}

\begin{proposition}\label{prop:ordinal-gamma}
Let $B$ be an ordinal sum given in \eqref{eq:ordinal_sum}. If $\tfrac12$ does not belong to any of the subintervals $(a_k, b_k)$, then 
\begin{equation} \label{eq:ordinal-gamma-1}
        \gamma(B) =  1 - \tfrac23\sum_{k=1}^n (b_k-a_k)^2(1-\phi(B_k)). 
\end{equation}  
If $\tfrac12 \in (a_j, b_j)$ for some $j \in \{1, 2, ..., n\}$ then
\begin{align*}
   \gamma(B) &= 4(b_j-a_j)^2\max\{0,c\}^2 + 4a_j - 4a_j^2 - \tfrac23\sum_{k=1}^n (b_k-a_k)^2(1-\phi(B_k)) \\
   &\qquad+ 4(b_j-a_j)^2\int_{\max\{0,c\}}^{1+\min\{0,c\}} B_j\big(t,1+c-t\big) \, dt,
\end{align*}
where $c=\frac{1-a_j-b_j}{b_j-a_j} \in (-1,1)$.
In particular, if $a_j + b_j = 1$ for some $j \in \{1, 2, ..., n\}$ (i.e., $\tfrac12$ lies at the middle of the subinterval $(a_j, b_j)$), then
\begin{equation} \label{eq:ordinal-gamma-3}
        \gamma(B) =  1 - (b_j-a_j)^2(1-\gamma(B_j)) - \tfrac23\sum_{\substack{k=1 \\ k\ne j}}^n (b_k-a_k)^2(1-\phi(B_k)). 
\end{equation}  
\end{proposition}

\begin{proof}
If $\tfrac12 \notin \cup_{k=1}^n(a_k, b_k)$ then $B(\tfrac12,\tfrac12)=\tfrac12$ and 
$\omega_B = \omega_M$. It follows that $\delta_{B^{\sigma_2}} = \delta_{M^{\sigma_2}} = \delta_{W}$ and $\phi(B^{\sigma_2}) = \phi(W) = -\tfrac12$.
By Lemma~\ref{lem:gamma-phi} and Proposition \ref{prop:ordinal-phi} we have
\begin{align*}
   \gamma(B) &= \tfrac23 \phi(B) - \tfrac23 \phi(B^{\sigma_2}) = \tfrac23 - \tfrac23\sum_{k=1}^n (b_k-a_k)^2(1-\phi(B_k)) + \tfrac13 \\
   &= 1 - \tfrac23\sum_{k=1}^n (b_k-a_k)^2(1-\phi(B_k)). 
\end{align*}
Suppose now that $\tfrac12 \in (a_j,b_j)$. We consider two cases. Suppose first that $a_j\ge 1-b_j$.
Then
$$\omega_B(u) = \begin{cases} u; & 0 \le u \le a_j,\\
		a_j + (b_j-a_j)B_j\big(\frac{u-a_j}{b_j-a_j}, \frac{1-u-a_j}{b_j-a_j}\big); & a_j < u \le 1-a_j,\\
		1-u; &  1-a_j < u \le 1,
    \end{cases}$$
and thus
$$\delta_{B^{\sigma_2}}(u) = \begin{cases} 0; & 0 \le u \le a_j,\\
		u - a_j - (b_j-a_j)B_j\big(\frac{u-a_j}{b_j-a_j}, \frac{1-u-a_j}{b_k-a_j}\big); & a_j < u \le 1-a_j,\\
		2u-1; &  1-a_j < u \le 1.
    \end{cases}$$
It follows that 
\begin{align*}
\phi(B^{\sigma_2}) &= 6 \int_0^1 \delta_{B^{\sigma_2}}(u) \, du -2 \\
&= -2 + 6\int_{a_j}^{1-a_j} \big(u - a_j - (b_j-a_j)B_j\big(\tfrac{u-a_j}{b_j-a_j}, \tfrac{1-u-a_j}{b_j-a_j}\big)\big) \, du + 6\int_{1-a_j}^1 (2u-1) \, du \\
&= 6a_j^2-6a_j+1 - 6(b_j-a_j)\int_{a_j}^{1-a_j} B_j\big(\tfrac{u-a_j}{b_j-a_j}, \tfrac{1-u-a_j}{b_j-a_j}\big) \, du \\
&= 6a_j^2-6a_j+1 - 6(b_j-a_j)^2\int_{0}^{\frac{1-2a_j}{b_j-a_j}} B_j\big(t, \tfrac{1-2a_j}{b_j-a_j}-t\big) \, dt. 
\end{align*}
Suppose now that $a_j < 1-b_j$. In this case 
$$\delta_{B^{\sigma_2}}(u) = \begin{cases} 0; & 0 \le u \le 1-b_j,\\
		u - a_j - (b_j-a_j)B_j\big(\frac{u-a_j}{b_j-a_j}, \frac{1-u-a_j}{b_k-a_j}\big); & 1-b_j < u \le b_j,\\
		2u-1; &  b_j < u \le 1,
    \end{cases}$$
and a similar calculation as above gives
\begin{align*}
\phi(B^{\sigma_2}) &= 12b_j-6b_j^2-12a_jb_j+6a_j-5 \\
&\qquad- 6(b_j-a_j)^2\int_{\frac{1-a_j-b_j}{b_j-a_j}}^{1} B_j\big(t, \tfrac{1-2a_j}{b_j-a_j}-t\big) \, dt \\
&=6a_j^2-6a_j+1-6(1-a_j-b_j)^2 - 6(b_j-a_j)^2\int_{\frac{1-a_j-b_j}{b_j-a_j}}^{1} B_j\big(t, \tfrac{1-2a_j}{b_j-a_j}-t\big) \, dt.
\end{align*}
Putting both cases together we obtain
\begin{align*}
    \phi(B^{\sigma_2}) &= 6a_j^2-6a_j+1-6\max\{0,1-a_j-b_j\}^2 - 6(b_j-a_j)^2\int_{\max\big\{0,\frac{1-a_j-b_j}{b_j-a_j}\big\}}^{\min\big\{1,\frac{1-2a_j}{b_j-a_j}\big\}} B_j\big(t, \tfrac{1-2a_j}{b_j-a_j}-t\big) \, dt\\
&=6a_j^2-6a_j+1-6(b_j-a_j)^2\max\{0,c\}^2- 6(b_j-a_j)^2\int_{\max\{0,c\}}^{1+\min\{0,c\}} B_j\big(t,1+c-t\big) \, dt,
\end{align*}
where $c=\frac{1-a_j-b_j}{b_j-a_j} \in (-1,1)$.
By Lemma~\ref{lem:gamma-phi} and Proposition \ref{prop:ordinal-phi} we have
\begin{align*}
   \gamma(B) &= \tfrac23 \phi(B) - \tfrac23 \phi(B^{\sigma_2}) \\
   &= \tfrac23 - \tfrac23\sum_{k=1}^n (b_k-a_k)^2(1-\phi(B_k)) - 4a_j^2 + 4a_j - \tfrac23+4(b_j-a_j)^2\max\{0,c\}^2  \\
   &\qquad+ 4(b_j-a_j)^2\int_{\max\{0,c\}}^{1+\min\{0,c\}} B_j\big(t,1+c-t\big) \, dt \\
   &= 4(b_j-a_j)^2\max\{0,c\}^2 + 4a_j - 4a_j^2 - \tfrac23\sum_{k=1}^n (b_k-a_k)^2(1-\phi(B_k)) \\
   &\qquad+ 4(b_j-a_j)^2\int_{\max\{0,c\}}^{1+\min\{0,c\}} B_j\big(t,1+c-t\big) \, dt. 
\end{align*}
In particular, if $a_j+b_j=1$ this simplifies into
$$\gamma(B) = 4a_j - 4a_j^2 - \tfrac23\sum_{k=1}^n (b_k-a_k)^2(1-\phi(B_k)) + 4(b_j-a_j)^2\int_0^1 B_j(t, 1-t) \, dt. $$
Since 
\begin{align*}
\int_0^1 B_j(t, 1-t) \, dt &= \int_0^1 \omega_{B_j}(t) \, dt = \int_0^1 (t-\delta_{B_j^{\sigma_2}}(t)) \, dt = \tfrac12-\int_0^1 \delta_{B_j^{\sigma_2}}(t) \, dt \\
&= \tfrac12-\tfrac16(\phi(B_j^{\sigma_2})+2) = \tfrac16(1-\phi(B_j^{\sigma_2})),
\end{align*}
we obtain
\begin{align*}
   \gamma(B) &= 4a_j - 4a_j^2 - \tfrac23\sum_{k=1}^n (b_k-a_k)^2(1-\phi(B_k)) + \tfrac23(b_j-a_j)^2(1-\phi(B_j^{\sigma_2})) \\
   &= 4a_j - 4a_j^2 + \tfrac23(b_j-a_j)^2(\phi(B_j)-\phi(B_j^{\sigma_2})) - \tfrac23\sum_{\substack{k=1 \\ k\ne j}}^n (b_k-a_k)^2(1-\phi(B_k))\\
   &= 4a_j - 4a_j^2 + (b_j-a_j)^2\gamma(B_j) - \tfrac23\sum_{\substack{k=1 \\ k\ne j}}^n (b_k-a_k)^2(1-\phi(B_k))\\
   &= 1 - (b_j-a_j)^2(1-\gamma(B_j)) - \tfrac23\sum_{\substack{k=1 \\ k\ne j}}^n (b_k-a_k)^2(1-\phi(B_k)),
\end{align*}
which finishes the proof.
\end{proof}

Next, we consider Blomqvist's beta of an ordinal sum. The proof is straightforward, so we omit it. 

\begin{proposition}\label{prop:ordinal-beta}
Let $B$ be an ordinal sum given in \eqref{eq:ordinal_sum}. If $\tfrac12$ does not belong to any of the subintervals $(a_k, b_k)$, then $\beta(B) = 1$. 
If $\tfrac12 \in (a_j, b_j)$ for some $j \in \{1, 2, ..., n\}$ then
$$ \beta(B) =  4a_j + 4(b_j-a_j)\delta_{B_j}\big(\tfrac{\frac12-a_j}{b_j-a_j}\big) -1. $$  
In particular, if $a_j + b_j = 1$ for some $j \in \{1, 2, ..., n\}$ (i.e., $\tfrac12$ lies in the middle of the subinterval $(a_j, b_j)$), then
$$ \beta(B) =  1 - (b_j-a_j)(1-\beta(B_j)). $$
\end{proposition}

Finally, we consider Chatterjee's xi of an ordinal sum of copulas. 

\begin{proposition}\label{prop:ordinal-xi}
   Chatterjee's xi of an ordinal sum $B$, given in \eqref{eq:ordinal_sum}, is
    \begin{equation} \label{eq:ordinal-xi}
        \xi(B) =  1 - \sum_{k=1}^n (b_k-a_k)^2(1-\xi(B_k)). 
    \end{equation}  
\end{proposition}

\begin{proof}
Note that the partial derivative $\partial_1B$ of copula $B$ with respect to the first variable equals
\begin{align*}
\partial_1B(u,v) &= \begin{cases}
		\partial_1B_k(\frac{u-a_k}{b_k-a_k},\frac{v-a_k}{b_k-a_k}); & (u,v) \in [a_k, b_k]^2, k=1, 2, \dots, n,\\
		1; & v \ge u, (u,v) \notin \cup_{k=1}^n[a_k, b_k]^2, \\
		0; & \text{otherwise},
    \end{cases} 
\end{align*}
so that 
\begin{align*}
\int_{\II^2} (\partial_1B(u,v))^2 \, du\, dv &= \sum_{k=1}^n  \int_{[a_k, b_k]^2} (\partial_1B_k(\tfrac{u-a_k}{b_k-a_k},\tfrac{v-a_k}{b_k-a_k}))^2 \, du\, dv + \int_{\{(u,v) \in \II^2 \mid v \ge u\} \setminus \cup_{k=1}^n[a_k, b_k]^2} du\, dv \\
&= \sum_{k=1}^n  \int_{[a_k, b_k]^2} (\partial_1B_k(\tfrac{u-a_k}{b_k-a_k},\tfrac{v-a_k}{b_k-a_k}))^2 \, du\, dv + \tfrac12 - \sum_{k=1}^n \tfrac12(b_k - a_k)^2 .
\end{align*}
Since 
\begin{align*}
\int_{[a_k, b_k]^2} (\partial_1B_k(\tfrac{u-a_k}{b_k-a_k},\tfrac{v-a_k}{b_k-a_k}))^2 \, du\, dv &= (b_k - a_k)^2 \int_{\II^2} (\partial_1B_k(s,t))^2 \, ds\, dt \\
&= (b_k - a_k)^2 \cdot\frac{\xi(B_k)+2}{6},
\end{align*}
we obtain
\begin{align*}
\xi(B) &= 6\int_{\II^2} (\partial_1B(u,v))^2 \, du\, dv -2 \\
&= \sum_{k=1}^n (b_k - a_k)^2(\xi(B_k)+2) + 3 - \sum_{k=1}^n 3(b_k - a_k)^2 - 2 \\
&= 1 - \sum_{k=1}^n (b_k-a_k)^2(1-\xi(B_k)),
\end{align*}
which finishes the proof.
\end{proof}

We finish this section by looking at a special simple ordinal sum, containing just one copula in the middle of the unit square, which will be needed later.

\begin{lemma}\label{lem:ordinal}
For any $a \in [0,\tfrac12]$ and copula $C\in\CC$ let 
$B_{a,C}$ be an ordinal sum of $\{C\}$ with respect to the interval $\{(a, 1-a)\}$. Then 
\begin{align*}
   \rho(B_{a,C}) &=  (1-2a)^3\rho(C) + 8a^3 - 12a^2 + 6a, \\
   \tau(B_{a,C}) &=  (1-2a)^2\tau(C) + 4a - 4a^2, \\
   \phi(B_{a,C}) &= (1-2a)^2\phi(C) + 4a - 4a^2, \\
   \gamma(B_{a,C}) &=  (1-2a)^2\gamma(C) + 4a - 4a^2, \\
   \beta(B_{a,C}) &=  (1-2a)\beta(C) + 2a, \qquad  \text{ and} \\
   \xi(B_{a,C}) &= (1-2a)^2\xi(C) + 4a - 4a^2.
\end{align*}
\end{lemma}

\begin{proof}
If $a=\tfrac12$ then $B_{a,C}=M$ and the formulas hold. So, we may assume without loss of generality that $a<\tfrac12$.
By Propositions \ref{prop:ordinal-rho},  \ref{prop:ordinal-tau},  \ref{prop:ordinal-phi},  \ref{prop:ordinal-gamma},  \ref{prop:ordinal-beta}, \ref{prop:ordinal-xi}, respectively,   we have
\begin{align*}
\rho(B_{a,C}) &= 1-(1-2a)^3(1-\rho(C)) = (1-2a)^3\rho(C) + 8a^3 - 12a^2 + 6a, \\
\tau(B_{a,C}) &= 1-(1-2a)^2(1-\tau(C)) = (1-2a)^2\tau(C) + 4a - 4a^2, \\
\phi(B_{a,C}) &= 1-(1-2a)^2(1-\phi(C)) = (1-2a)^2\phi(C) + 4a - 4a^2, \\
\gamma(B_{a,C}) &= 1-(1-2a)^2(1-\gamma(C)) = (1-2a)^2\gamma(C) + 4a - 4a^2, \\
\beta(B_{a,C}) &= 1-(1-2a)(1-\beta(C)) = (1-2a)\beta(C) + 2a, \qquad \text{and} \\
\xi(B_{a,C}) &= 1-(1-2a)^2(1-\xi(C)) = (1-2a)^2\xi(C) + 4a - 4a^2, 
\end{align*}
respectively, which finishes the proof.
\end{proof}

\section{Exact region phi-gamma-tau}\label{sec:region}

In this section we find the exact region determined by Spearman's footrule, Gini's gamma and Kendall's tau
$$\Omega_{\phi, \gamma, \tau} = \big\{( \phi(C), \gamma(C), \tau(C)) \in [-\tfrac12, 1] \times [-1, 1] \times [-1, 1] \mid C \in \CC\big\}.$$
The first lemma gives the bounds for Kendall's tau of a copula, if its Spearman's footrule and Gini's gamma are known.

\begin{lemma}\label{lem:bounds}
For any copula $C\in\CC$ we have
$$ \max\big\{\tfrac43\phi(C)-\tfrac13,\ -\tfrac23\phi(C)+\gamma(C)-\tfrac13 \big\} \le \tau(C) \le \min\big\{\tfrac23\phi(C)+\tfrac13,\ -\tfrac43\phi(C)+2\gamma(C)+\tfrac13 \big\}.$$
\end{lemma}

\begin{proof}
By \eqref{eq:tau-phi} we need to prove only
$$ -\tfrac23\phi(C)+\gamma(C)-\tfrac13 \le \tau(C) \le -\tfrac43\phi(C)+2\gamma(C)+\tfrac13.$$
For any copula $C\in \CC$ it follows from Lemma~\ref{lem:gamma-phi} that
\begin{equation}
    \phi(C^{\sigma_2}) = \phi(C) - \tfrac32 \gamma(C).
\end{equation}
Using \eqref{eq:tau-phi} again, we infer
$$\tau(C) = - \tau(C^{\sigma_2}) \le -(\tfrac43\phi(C^{\sigma_2})-\tfrac13) = -\tfrac43(\phi(C) - \tfrac32 \gamma(C)) + \tfrac13  = -\tfrac43\phi(C)+2\gamma(C)+\tfrac13,$$
and
$$\tau(C) = - \tau(C^{\sigma_2}) \ge -(\tfrac23\phi(C^{\sigma_2})+\tfrac13) = -\tfrac23(\phi(C) - \tfrac32 \gamma(C)) - \tfrac13  = -\tfrac23\phi(C)+\gamma(C)-\tfrac13,$$
which finishes the proof.
\end{proof}

We define a polyhedron $\Omega$ in $\RR^3$ by
\begin{align} 
   \Omega = \{(\phi, \gamma, \tau) \in \RR^3 \mid 
   -2\phi &\le 1, \nonumber\\
    - 8\phi + 6\gamma &\leq 1, \nonumber\\
   - 2\phi + 3\gamma &\leq 1, \nonumber\\
   4\phi-3\tau &\le 1, \label{eq:Omega}\\
   - 2\phi + 3\tau  &\le 1, \nonumber\\
   -2\phi+3\gamma-3\tau &\le 1, \nonumber\\ 
   4\phi-6\gamma + 3\tau &\le 1\}.\nonumber
\end{align}
If follows from the fact that $\phi(C) \ge -\tfrac12$, equation \eqref{eq:gamma-phi} and Lemma \ref{lem:bounds} that for any copula $C \in \CC$ the point $(\phi(C), \gamma(C), \tau(C))$ lies in the polyhedron $\Omega$. Note that the polyhedron $\Omega$ has 6 vertices
$$P_1(-\tfrac12,-1,-1),\ P_2(-\tfrac12,-\tfrac12,-\tfrac12),\ P_3(-\tfrac12,-\tfrac12,0),\ P_4(1,1,1),\ P_5(\tfrac14,\tfrac12,0),\ P_6(\tfrac14,\tfrac12,\tfrac12), $$
7 faces
\begin{align*}
  F_1 &:=  P_1 P_2 P_3 = \{ (\phi,\gamma, \tau) \in \Omega \mid -2\phi = 1 \}, \\
  F_2 &:=  P_1 P_2 P_5 = \{ (\phi,\gamma, \tau) \in \Omega \mid -2\phi+3\gamma-3\tau = 1 \}, \\
  F_3 &:=  P_1 P_3 P_4 = \{ (\phi,\gamma, \tau) \in \Omega \mid 4\phi-6\gamma + 3\tau = 1  \}, \\
  F_4 &:=  P_1 P_4 P_5 = \{ (\phi,\gamma, \tau) \in \Omega \mid 4\phi-3\tau = 1  \}, \\
  F_5 &:=  P_2 P_3 P_5 P_6 = \{ (\phi,\gamma, \tau) \in \Omega \mid - 8\phi + 6\gamma = 1  \}, \\
  F_6 &:=  P_3 P_4 P_6 = \{ (\phi,\gamma, \tau) \in \Omega \mid - 2\phi + 3\tau = 1  \}, \\
  F_7 &:=  P_4 P_5 P_6 = \{ (\phi,\gamma, \tau) \in \Omega \mid - 2\phi + 3\gamma = 1  \},
\end{align*}
and 11 edges
\begin{align*}
    P_1 P_2 &= \{ (\phi,\gamma, \tau) \in \Omega \mid -2\phi =1,\ -2\phi+3\gamma-3\tau = 1 \}, \\
    P_1 P_3 &= \{ (\phi,\gamma, \tau) \in \Omega \mid -2\phi = 1, \  4\phi-6\gamma + 3\tau=1  \}, \\
    P_1 P_4 &= \{ (\phi,\gamma, \tau) \in \Omega \mid 4\phi-3\tau = 1, \ 4\phi-6\gamma + 3\tau=1  \}, \\
    P_1 P_5 &= \{ (\phi,\gamma, \tau) \in \Omega \mid 4\phi-3\tau = 1, \ -2\phi+3\gamma-3\tau = 1 \}, \\
    P_2 P_3 &= \{ (\phi,\gamma, \tau) \in \Omega \mid \phi = -\tfrac12, \  \gamma= -\tfrac12  \}, \\
    P_2 P_5 &= \{ (\phi,\gamma, \tau) \in \Omega \mid -8\phi + 6\gamma = 1, \ -2\phi+3\gamma-3\tau = 1  \}, \\
    P_3 P_4 &= \{ (\phi,\gamma, \tau) \in \Omega \mid - 2\phi + 3\tau  = 1, \ 4\phi-6\gamma + 3\tau = 1  \}, \\
    P_3 P_6 &= \{ (\phi,\gamma, \tau) \in \Omega \mid - 8\phi + 6\gamma = 1, \ - 2\phi + 3\tau = 1  \}, \\
    P_4 P_5 &= \{ (\phi,\gamma, \tau) \in \Omega \mid - 2\phi + 3\gamma = 1, \ 4\phi-3\tau = 1  \}, \\
    P_4 P_6 &= \{ (\phi,\gamma, \tau) \in \Omega \mid - 2\phi + 3\gamma = 1, \ - 2\phi + 3\tau = 1  \}, \\
    P_5 P_6 &= \{ (\phi,\gamma, \tau) \in \Omega \mid \phi =\tfrac14, \ \gamma =\tfrac12  \}.
\end{align*}
The polyhedron $\Omega$ is depicted in Figure~\ref{fig:polyhedron} from two distinct points of view.
\begin{figure}
    \centering
    \includegraphics[width=0.48\linewidth]{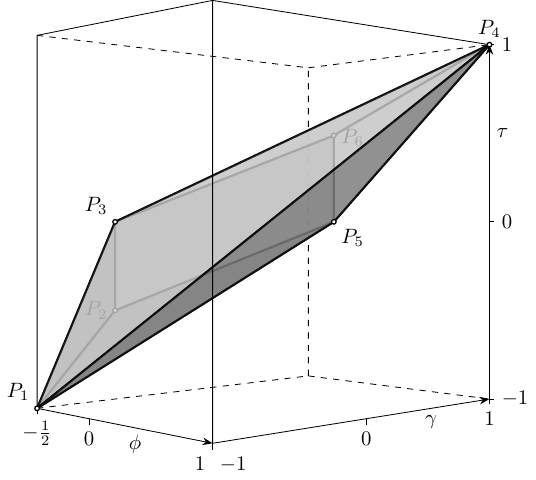}
    \includegraphics[width=0.465\linewidth]{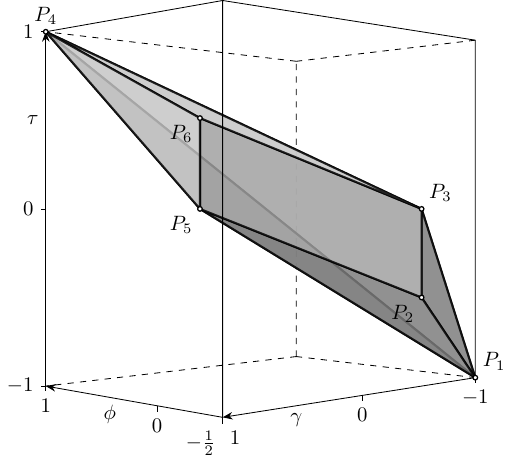}
    \caption{The polyhedron $\Omega$ defined in \eqref{eq:Omega}. Theorem \ref{thm:main} shows that $\Omega$ is precisely the region determined by $\phi$, $\gamma$ and $\tau$.}
    \label{fig:polyhedron}
\end{figure}

Note also that the projection of the polyhedron $\Omega$ to the $\phi\gamma$-plane is the quadrilateral 
$P'_1(-\tfrac12,-1)P'_4(1,1)P'_5(\tfrac14,\tfrac12)P'_2(-\tfrac12,-\tfrac12)$, the exact region determined by Spearman's footrule and Gini's gamma (see \cite{KoBuMo}), 
the projection of the polyhedron $\Omega$ to the $\phi\tau$-plane is the triangle
$P''_1(-\tfrac12,-1)P''_4(1,1)P''_3(-\tfrac12,0)$, the exact region determined by Spearman's footrule and Kendall's tau (see \cite{KoBuSt}), 
and the projection of the polyhedron $\Omega$ to the $\gamma\tau$-plane is the quadrilateral 
$P'''_1(-1,-1)P'''_3(-\tfrac12,0)P'''_4(1,1)P'''_5(\tfrac12,0)$, the exact region determined by Gini's gamma and Kendall's tau (see \cite{KoBuSt}).

Since $\gamma$ and $\tau$ are concordance measures, we have $$\gamma(C^{\sigma_2})=-\gamma(C) \qquad \text{and} \qquad \tau(C^{\sigma_2})=-\tau(C).$$
Furthermore, Lemma~\ref{lem:gamma-phi} implies $\phi(C^{\sigma_2})=\phi(C)-\tfrac32 \gamma(C)$. Thus, the reflection mapping $C \mapsto C^{\sigma_2}$ induces a linear involution
$$\mathcal{A}\colon (\phi, \gamma, \tau) \mapsto (\phi - \tfrac32 \gamma, -\gamma, -\tau).$$
In particular, for any copula $C\in \CC$ we have
\begin{equation}\label{eq:mappingA}
  \mathcal{A}(\phi(C), \gamma(C), \tau(C)) = (\phi(C^{\sigma_2}), \gamma(C^{\sigma_2}), \tau(C^{\sigma_2})).  
\end{equation}

The following lemma is easy to check.
\begin{lemma}\label{lem:linear_map_A}
Linear mapping $\mathcal{A}$ has the following properties
\begin{align*}
\mathcal{A}(P_1)=P_4, && \mathcal{A}(F_1)=F_7,\\
\mathcal{A}(P_2)=P_6, && \mathcal{A}(F_2)=F_6,\\
\mathcal{A}(P_3)=P_5, && \mathcal{A}(F_3)=F_4,\\
\mathcal{A}(P_4)=P_1, && \mathcal{A}(F_4)=F_3,\\
\mathcal{A}(P_5)=P_3, && \mathcal{A}(F_5)=F_5,\\
\mathcal{A}(P_6)=P_2, && \mathcal{A}(F_6)=F_2, \\
&& \mathcal{A}(F_7)=F_1.
\end{align*}
Moreover, mapping $\mathcal{A}$ preserves polyhedron $\Omega$.
\end{lemma}

In the following lemmas we show that points on the faces $F_2$, $F_3$, $F_4$, and $F_6$ are attained by some copulas.

\begin{lemma}\label{lem:F6}
For any point $T$ on the face $F_6$ of the polyhedron $\Omega$
$$F_6 = P_3P_4P_6 = \big\{(\phi, \gamma, \tau) \in \RR^3 \mid -\tfrac12 \le \phi \le 1,\ \phi \le \gamma \le \min\{\tfrac43 \phi+\tfrac16, \tfrac23 \phi+\tfrac13\},\ \tau = \tfrac23\phi+\tfrac13\big\}$$
there exists a copula $C\in\CC$ such that $(\phi(C), \gamma(C), \tau(C)) = T$.
\end{lemma}

\begin{proof}
Let $b \in [0, \tfrac14]$  and let $C_b$ be a shuffle of $M$
$$C_b = M\big(6, (b, \tfrac12 - b, \tfrac12, \tfrac12 + b, 1-b), (3,5,1,6,2,4), (1, 1, 1, 1, 1, 1)\big).$$
The copula $C_b$ is defined by the measure preserving bijection 
$$ h_{C_b}(u) =  \begin{cases}
            u + \tfrac12 - b; &  0 \le u \le b, \\
            u + \tfrac12; &  b < u \le \tfrac12 - b, \\
            u - \tfrac12 + b; &  \tfrac12 - b < u \le \tfrac12, \\
            u + \tfrac12 - b; &  \tfrac12 < u \le \tfrac12 + b, \\
            u - \tfrac12; &  \tfrac12 + b < u \le 1-b, \\
            u - \tfrac12 + b; &  1-b < u \le 1.
            \end{cases}
$$
We have 
$$ C_b(u,h_{C_b}(u)) =  \begin{cases}
            u; &  0 \le u \le \tfrac12 - b, \\
            h_{C_b}(u); &  \tfrac12 - b < u \le \tfrac12, \\
            u; &  \tfrac12 < u \le \tfrac12 + b, \\
            h_{C_b}(u); &  \tfrac12 + b < u \le 1.
            \end{cases}
$$
Using equation \eqref{eq:tau-h}, we obtain
$$  \tau(C_b) = 8b^2. $$
Furthermore, we have
$$ \delta_{C_b}(u) =  \begin{cases}
            0; &  0 \le u \le \tfrac12-b, \\
            2u - 1 + 2b; &  \tfrac12-b < u \le \tfrac12, \\
            2b; &  \tfrac12 < u \le \tfrac12+b, \\
            2u-1; &  \tfrac12+b < u \le 1, 
            \end{cases}
$$
and
$$ \omega_{C_b}(u) =  \begin{cases}
            u; &  0 \le u \le \tfrac14, \\
            \tfrac12 - u; &  \tfrac14 < u \le \tfrac12 - b, \\     
            u + 2b - \tfrac12; &  \tfrac12-b < u \le \tfrac12, \\
            2b + \tfrac12 - u; &  \tfrac12 < u \le \tfrac12 + b, \\
            u - \tfrac12; &  \tfrac12 + b < u \le \tfrac34, \\
            1-u; &  \tfrac34 < u \le 1, 
            \end{cases}
$$
Using equations \eqref{phi}  and \eqref{gamma}, we obtain
$$  \phi(C_b) = 12b^2-\tfrac12, \quad \text{and} \quad \gamma(C_b) = 16b^2-\tfrac12. $$
The mass distribution of copula $C_b$ and the graphs of functions $\delta_{C_b}$ and $\omega_{C_b}$ are shown in Figure~\ref{fig:lem F6}.

Now let $A_{a,b} = B_{a,C_b}$ be an ordinal sum of $\{C_b\}$ with respect to the interval $\{(a, 1-a)\}$. By Lemma \ref{lem:ordinal} we have
\begin{align*}
   \tau(A_{a,b}) &= 8(1-2a)^2b^2-4a^2+4a, \\
   \phi(A_{a,b}) &= 12(1-2a)^2b^2-6a^2+6a-\tfrac12, \text{ and} \\
   \gamma(A_{a,b}) &= 16(1-2a)^2b^2-6a^2+6a-\tfrac12.
\end{align*}
Note that 
\begin{align*}
  \tfrac23\phi(A_{a,b})+\tfrac13 &= \tfrac23\big(12(1-2a)^2b^2-6a^2+6a-\tfrac12\big)+\tfrac13  \\
  &= 8(1-2a)^2b^2-4a^2+4a = \tau(A_{a,b}).
\end{align*}
Since the point $(\phi(A_{a,b}), \gamma(A_{a,b}), \tau(A_{a,b}))$ lies in the polyhedron $\Omega$, this implies that it lies on the face $F_6$ for any $a\in [0, \tfrac12]$ and $b\in [0, \tfrac14]$.
Furthermore, $A_{\frac12,b} = M$, so that the point $(\phi(A_{\frac12,b}), \gamma(A_{\frac12,b}), \tau(A_{\frac12,b}))$ is the point $P_4$ of the polyhedron $\Omega$ for any $b\in [0, \tfrac14]$, 
$$   \tfrac43\phi(A_{0,b}) + \tfrac16  = \tfrac43\phi(C_b) + \tfrac16 = \tfrac43(12b^2-\tfrac12) + \tfrac16 = 16b^2-\tfrac12 = \gamma(C_b) = \gamma(A_{0,b}),$$
so that the point $(\phi(A_{0,b}), \gamma(A_{0,b}), \tau(A_{0,b}))$ lies on the edge $P_3P_6$ of the polyhedron $\Omega$ for any $b\in [0, \tfrac14]$, 
$$  \phi(A_{a,0}) = -6a^2+6a-\tfrac12 = \gamma(A_{a,0}),$$
so that the point $(\phi(A_{a,0}), \gamma(A_{a,0}), \tau(A_{a,0}))$ lies on the edge $P_3P_4$ of the polyhedron $\Omega$ for any $a\in [0, \tfrac12]$, and
$$  \tfrac23\phi(A_{a,\frac14})+\tfrac13 = \tfrac23(-3a^2+3a+\tfrac14) + \tfrac13 = -2a^2+2a+\tfrac12 = \gamma(A_{a,\frac14}),$$
so that the point $(\phi(A_{a,\frac14}), \gamma(A_{a,\frac14}), \tau(A_{a,\frac14}))$ lies on the edge $P_4P_6$ of the polyhedron $\Omega$ for any $a\in [0, \tfrac12]$. For any fixed $a\in [0, \tfrac12]$ the set of points 
$$\big\{(\phi(A_{a,b}), \gamma(A_{a,b}), \tau(A_{a,b})) \mid b\in [0, \tfrac14]\big\}$$ 
is a curve lying on the face $F_6$ and connecting a point on the edge $P_3P_4$ and a point on the edge $P_4P_6$. Since $\phi(A_{a,b}), \gamma(A_{a,b}), \tau(A_{a,b})$ are continuous as functions of $a,b$, this means that for any point $T$ on the face $F_6$ there exist $a\in [0, \tfrac12]$ and $b\in [0, \tfrac14]$ such that $(\phi(A_{a,b}), \gamma(A_{a,b}), \tau(A_{a,b})) = T$.
\end{proof}

\begin{figure}[ht]
\centering
\includegraphics{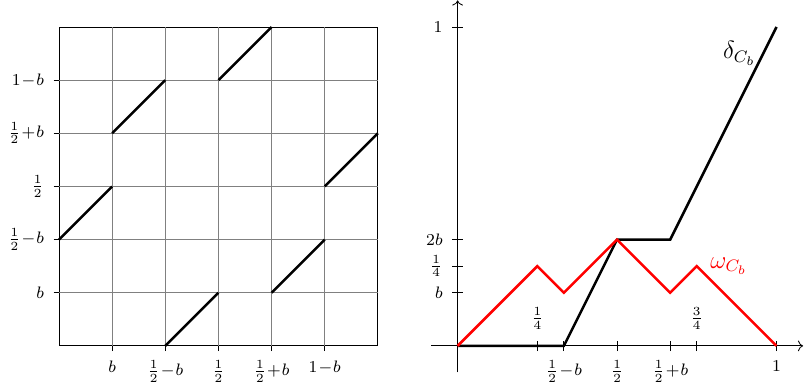}
    \caption{The mass distribution of copula $C_b$ (left) and the graphs of functions $\delta_{C_b}$ and $\omega_{C_b}$ (right) from the proof of Lemma \ref{lem:F6}.}
    \label{fig:lem F6}
\end{figure}

\begin{lemma}\label{lem:F2}
For any point $T$ on the face $F_2$ of the polyhedron $\Omega$
$$F_2 = P_1P_2P_5 = \big\{(\phi, \gamma, \tau) \in \RR^3 \mid  -\tfrac12 \le \phi \le \tfrac14,\ 2\phi \le \gamma \le \tfrac43 \phi+\tfrac16,\ \tau = -\tfrac23\phi+\gamma-\tfrac13\big\}$$
there exists a copula $C\in\CC$ such that $(\phi(C), \gamma(C), \tau(C)) = T$.
\end{lemma}

\begin{proof}
Let $T(\phi, \gamma, \tau)$ be any point in $F_2$. Then the point $T'(\phi-\tfrac32 \gamma, -\gamma, -\tau) = \mathcal{A}(T)$ lies on the face $F_6$, by Lemma \ref{lem:linear_map_A}. By applying Lemma \ref{lem:F6} to the point $T'$, we can find a copula $C\in\CC$ such that $(\phi(C), \gamma(C), \tau(C)) = T'$. By equation~\eqref{eq:mappingA} we get $T=(\phi(C^{\sigma_2}), \gamma(C^{\sigma_2}), \tau(C^{\sigma_2}))$, which finishes the proof.
\end{proof}

\begin{lemma}\label{lem:F4}
For any point $T$ on the face $F_4$ of the polyhedron $\Omega$
$$F_4 = P_1P_4P_5 = \big\{(\phi, \gamma, \tau) \in \RR^3 \mid  -\tfrac12 \le \phi \le 1,\ \tfrac43 \phi - \tfrac13 \le \gamma \le \min\{2\phi, \tfrac23 \phi+\tfrac13\},\ \tau = \tfrac43 \phi-\tfrac13\big\}$$
there exists a copula $C\in\CC$ such that $(\phi(C), \gamma(C), \tau(C)) = T$.
\end{lemma}

\begin{proof}
Let $b \in [0, \tfrac12]$  and let $D_b$ be a shuffle of $M$
$$D_b = M\big(3, (b, 1-b), (1, 2, 3), (-1, 1, -1)\big).$$
The copula $D_b$ is defined by the measure preserving bijection 
$$ h_{D_b}(u) =  \begin{cases}
            b - u; &  0 \le u \le b, \\
            u; & b < u \le 1-b, \\
            2 - b - u; &  1-b < u \le 1.
            \end{cases}
$$
We have 
$$ D_b(u,h_{D_b}(u)) =  \begin{cases}
            0; &  0 \le u \le b, \\
            u; &  b < u \le 1-b, \\
            1-b; &  1 - b < u \le 1.
            \end{cases}
$$
Using equation \eqref{eq:tau-h}, we obtain
$$  \tau(D_b) = 1-4b^2. $$
Furthermore, we have
$$ \delta_{D_b}(u) =  \begin{cases}
            0; &  0 \le u \le \tfrac{b}{2}, \\
            2u - b; &  \tfrac{b}{2} < u \le b, \\
            u; & b < u \le 1-b, \\
            1-b; &  1-b < u \le 1-\tfrac{b}{2}, \\
            2u-1; &  1-\tfrac{b}{2} < u \le 1, 
            \end{cases}
$$
and
$$ \omega_{D_b}(u) =  \begin{cases}
            u; &  0 \le u \le \tfrac12, \\
            1-u; &  \tfrac12 < u \le 1, 
            \end{cases}
$$
Using equations \eqref{phi}  and \eqref{gamma}, we obtain
$$  \phi(D_b) = 1-3b^2, \quad \text{and} \quad \gamma(D_b) = 1- 2b^2. $$
The mass distribution of copula $D_b$ and the graphs of functions $\delta_{D_b}$ and $\omega_{D_b}$ are shown in Figure~\ref{fig:lem F4}.

Now let $E_{a,b} = B_{a,D_b^{\sigma_2}}$ be an ordinal sum of $\{D_b^{\sigma_2}\}$ with respect to the interval $\{(a, 1-a)\}$ and $F_{a,b} = E_{a,b}^{\sigma_2}$. Using Lemmas \ref{lem:ordinal} and \ref{lem:gamma-phi} we obtain
\begin{align*}
   \tau(F_{a,b}) &=  (1-2a)^2\tau(D_b) + 4(a^2 - a) = -4(1-2a)^2b^2 + 8a^2 - 8a + 1, \\
   \phi(F_{a,b}) &= (1-2a)^2\phi(D_b) + 2(a^2 - a) = -3(1-2a)^2b^2 + 6a^2 - 6a + 1, \text{ and} \\
   \gamma(F_{a,b}) &=  (1-2a)^2\gamma(D_b) + 4(a^2 - a) = -2(1-2a)^2b^2 + 8a^2 - 8a + 1.
\end{align*}
Note that 
\begin{align*}
  \tfrac43\phi(F_{a,b})-\tfrac13 &= \tfrac43(-3(1-2a)^2b^2 + 6a^2 - 6a + 1)-\tfrac13  \\
  &= -4(1-2a)^2b^2 + 8a^2 - 8a + 1 = \tau(F_{a,b}).
\end{align*}
Since the point $(\phi(F_{a,b}), \gamma(F_{a,b}), \tau(F_{a,b}))$ lies in the polyhedron $\Omega$, this implies that it lies on the face $F_4$ for any $a, b\in [0, \tfrac12]$.
Furthermore, $F_{\frac12,b} = W$, so that the point $(\phi(F_{\frac12,b}), \gamma(F_{\frac12,b}), \tau(F_{\frac12,b}))$ is the point $P_1$ of the polyhedron $\Omega$ for any $b\in [0, \tfrac12]$, 
$$   \tfrac23\phi(F_{0,b}) + \tfrac13  = \tfrac23\phi(D_b) + \tfrac13 = \tfrac23(1-3b^2) + \tfrac13 = 1-2b^2 = \gamma(D_b) = \gamma(F_{0,b}),$$
so that the point $(\phi(F_{0,b}), \gamma(F_{0,b}), \tau(F_{0,b}))$ lies on the edge $P_4P_5$ of the polyhedron $\Omega$ for any $b\in [0, \tfrac12]$, 
$$ \tfrac43\phi(F_{a,0}) - \tfrac13 = \tfrac43(6a^2 - 6a + 1) - \tfrac13 = 8a^2 - 8a + 1 = \gamma(F_{a,0}),$$
so that the point $(\phi(F_{a,0}), \gamma(F_{a,0}), \tau(F_{a,0}))$ lies on the edge $P_1P_4$ of the polyhedron $\Omega$ for any $a\in [0, \tfrac12]$, and
$$  2\phi(F_{a,\frac12})= 6a^2 - 6a + \tfrac12 = \gamma(F_{a,\frac12}),$$
so that the point $(\phi(F_{a,\frac12}), \gamma(F_{a,\frac12}), \tau(F_{a,\frac12}))$ lies on the edge $P_1P_5$ of the polyhedron $\Omega$ for any $a\in [0, \tfrac12]$. For any fixed $a\in [0, \tfrac12]$ the set of points 
$$\big\{(\phi(F_{a,b}), \gamma(F_{a,b}), \tau(F_{a,b})) \mid b\in [0, \tfrac12]\big\}$$ 
is a curve lying on the face $F_4$ and connecting a point on the edge $P_1P_4$ and a point on the edge $P_1P_5$. Since $\phi(F_{a,b}), \gamma(F_{a,b}), \tau(F_{a,b})$ are continuous as functions of $a,b$, this means that for any point $T$ on the face $F_4$ there exist $a,b\in [0, \tfrac12]$ such that $(\phi(F_{a,b}), \gamma(F_{a,b}), \tau(F_{a,b})) = T$.
\end{proof}

\begin{figure}[ht]
    \centering
\includegraphics{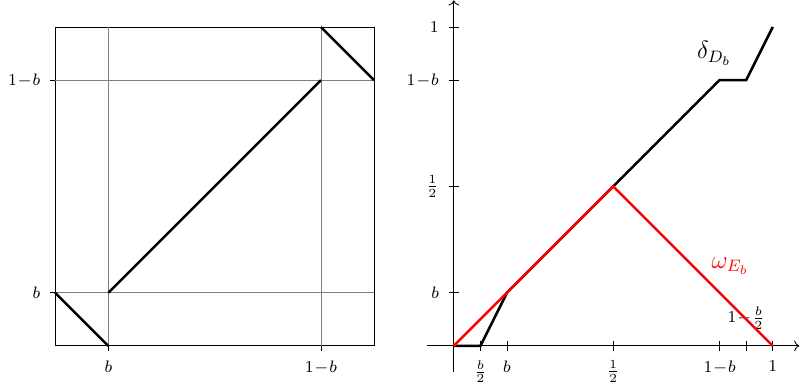}
    \caption{The mass distribution of copula $D_b$ (left) and the graphs of functions $\delta_{D_b}$ and $\omega_{D_b}$ (right) from the proof of Lemma \ref{lem:F4}.}
    \label{fig:lem F4}
\end{figure}

\begin{lemma}\label{lem:F3}
For any point $T$ on the face $F_3$ of the polyhedron $\Omega$
$$F_3 = P_1P_3P_4 = \big\{(\phi, \gamma, \tau) \in \RR^3 \mid -\tfrac12 \le \phi \le 1,\ \tfrac43 \phi - \tfrac13 \le \gamma \le \phi,\ \tau = -\tfrac43 \phi + 2\gamma +\tfrac13\big\}$$
there exists a copula $C\in\CC$ such that $(\phi(C), \gamma(C), \tau(C)) = T$.
\end{lemma}

\begin{proof}
By Lemma \ref{lem:linear_map_A}, linear mapping $\mathcal{A}$ maps $F_3$ to $F_4$, so for any point $T(\phi, \gamma, \tau)$ in $F_3$ the point $T'(\phi-\tfrac32 \gamma, -\gamma, -\tau) = \mathcal{A}(T)$ lies on the face $F_4$. By Lemma \ref{lem:F4}, there exists a copula $C\in\CC$ such that $(\phi(C), \gamma(C), \tau(C)) = T'$. Thus, by equation~\eqref{eq:mappingA} we have $(\phi(C^{\sigma_2}), \gamma(C^{\sigma_2}), \tau(C^{\sigma_2})) = T$.
\end{proof}

We are now ready to prove our main theorem.

\begin{theorem} \label{thm:main}
The exact region determined by Spearman's footrule, Gini's gamma and Kendall's tau
$$\Omega_{\phi, \gamma, \tau} = \big\{( \phi(C), \gamma(C), \tau(C)) \in [-\tfrac12, 1] \times [-1, 1] \times [-1, 1] \mid C \in \CC\big\}$$
is the polyhedron $\Omega$, defined in \eqref{eq:Omega}.
\end{theorem}

\begin{proof}
If follows from equation \eqref{eq:gamma-phi} and Lemma \ref{lem:bounds} that $\Omega_{\phi, \gamma, \tau} \subseteq \Omega$. We need to prove that for any point $T \in \Omega$ there exists a copula $C\in\CC$ such that $(\phi(C), \gamma(C), \tau(C)) = T$. Let $T(a,b,c) \in \Omega$.
The vertical line $\{(a,b, \tau) \mid \tau \in \RR\}$ intersects polyhedron $\Omega$ in the line segment $AB$, where the point $A(a,b, \tau_0)$ lies in the union of faces $F_2$ and $F_4$ and the point $B(a,b, \tau_1)$ lies in the union of faces $F_3$ and $F_6$. By Lemmas \ref{lem:F6}, \ref{lem:F2}, \ref{lem:F4}, and \ref{lem:F3}, there exist copulas $C_0$ and $C_1$ such that $(\phi(C_0), \gamma(C_0), \tau(C_0)) = A(a,b, \tau_0)$ and $(\phi(C_1), \gamma(C_1), \tau(C_1)) = B(a,b, \tau_1)$.
Now for any $t \in \II$ let $C_t$ be a convex combination of copulas $C_0$ and $C_1$, i.e., $C_t = tC_1 + (1-t)C_0$. We have $\phi(C_t) = t\phi(C_1) + (1-t)\phi(C_0) = a$ and $\gamma(C_t) = t\gamma(C_1) + (1-t)\gamma(C_0) = b$. Furthermore, $\tau(C_0) = \tau_0 \le c \le \tau_1 = \tau(C_1)$ and $\tau(C_t)$ is continuous as a function of $t$, so there exists $t \in \II$ such that $\tau(C_t) = c$, hence $T=(\phi(C_t), \gamma(C_t), \tau(C_t))$.
\end{proof}

Lemmas \ref{lem:F6}, \ref{lem:F2}, \ref{lem:F4}, and \ref{lem:F3} show that any point in the union of faces $F_2$, $F_3$, $F_4$, and $F_6$ can be attained by a shuffle of $M$. The following examples demonstrate that also the points on the remaining faces $F_1$, $F_5$, and $F_7$ are attained by shuffles of $M$.

\begin{example}\label{ex:F7}
Let $b \in [0, \tfrac14]$  and let $G_b$ be a shuffle of $M$
$$G_b = M\big(6, (b, \tfrac12 - b, \tfrac12, \tfrac12 + b, 1-b), (3, 2, 1, 6, 5, 4), (1, -1, 1, 1, -1, 1)\big).$$
The copula $G_b$ is defined by the measure preserving bijection 
$$ h_{G_b}(u) =  \begin{cases}
            u + \tfrac12 - b; &  0 \le u \le b, \\
            \tfrac12 - u; &  b < u \le \tfrac12 - b, \\
            u - \tfrac12 + b; &  \tfrac12 - b < u \le \tfrac12, \\
            u + \tfrac12 - b; &  \tfrac12 < u \le \tfrac12 + b, \\
            \tfrac32 - u; &  \tfrac12 + b < u \le 1-b, \\
            u - \tfrac12 + b; &  1-b < u \le 1.
            \end{cases}
$$
We have 
$$ G_b(u,h_{G_b}(u)) =  \begin{cases}
            u; &  0 \le u \le b, \\
            0; &  b < u \le \tfrac12 - b, \\
            h_{G_b}(u); &  \tfrac12 - b < u \le \tfrac12, \\
            u; &  \tfrac12 < u \le \tfrac12 + b, \\
            \tfrac12; &  \tfrac12 + b < u \le 1 - b, \\
            h_{G_b}(u); &  1 - b < u \le 1.
            \end{cases}
$$
Using equation \eqref{eq:tau-h}, we obtain
$$  \tau(G_b) = 8b^2. $$
Furthermore, we have
$$ \delta_{G_b}(u) =  \begin{cases}
            0; &  0 \le u \le \tfrac14, \\
            2u - \tfrac12; &  \tfrac14 < u \le \tfrac12, \\
            \tfrac12; &  \tfrac12 < u \le \tfrac34, \\
            2u-1; &  \tfrac34 < u \le 1, 
            \end{cases}
$$
and
$$ \omega_{G_b}(u) = \begin{cases}
            u; &  0 \le u \le \tfrac12, \\
            1-u; &  \tfrac12 < u \le 1, 
            \end{cases}
$$
The mass distribution of copula $G_b$ and the graphs of functions $\delta_{G_b}$ and $\omega_{G_b}$ are shown in Figure~\ref{fig:ex F7}.
Using equations \eqref{phi}  and \eqref{gamma}, we obtain
$$  \phi(G_b) = \tfrac14, \quad \text{and} \quad \gamma(G_b) = \tfrac12. $$

Now, for any $a \in [0, \tfrac12]$ let $H_{a,b} = B_{a,G_b}$ be an ordinal sum of $\{G_b\}$ with respect to the interval $\{(a, 1-a)\}$. 
By Lemma \ref{lem:ordinal} we have
\begin{align*}
   \tau(H_{a,b}) &= 8(1-2a)^2b^2-4a^2+4a, \\
   \phi(H_{a,b}) &= -3a^2+3a+\tfrac14, \text{ and} \\
   \gamma(H_{a,b}) &= -2a^2+2a+\tfrac12.
\end{align*}
Note that 
$$ \tfrac23\phi(H_{a,b})+\tfrac13 = \tfrac23(-3a^2+3a+\tfrac14)+\tfrac13  
= -2a^2+2a+\tfrac12 = \gamma(H_{a,b}),$$
so the point $(\phi(H_{a,b}), \gamma(H_{a,b}), \tau(H_{a,b}))$ lies on the face $F_7$ of the polyhedron $\Omega$ for any $a\in [0, \tfrac12]$ and $b\in [0, \tfrac14]$.
Furthermore, $H_{\frac12,b} = M$, so that the point $(\phi(H_{\frac12,b}), \gamma(H_{\frac12,b}), \tau(H_{\frac12,b}))$ is the point $P_4$ of the polyhedron $\Omega$ for any $b\in [0, \tfrac14]$, 
$$ \phi(H_{0,b}) = \phi(G_b) = \tfrac14 \quad \text{and} \quad \gamma(H_{0,b}) = \gamma(G_b) = \tfrac12,$$
so that the point $(\phi(H_{0,b}), \gamma(H_{0,b}), \tau(H_{0,b}))$ lies on the edge $P_5P_6$ of the polyhedron $\Omega$ for any $b\in [0, \tfrac14]$, 
$$  \tfrac43\phi(H_{a,0})-\tfrac13 = \tfrac43(-3a^2+3a+\tfrac14)-\tfrac13 = -4a^2+4a = \tau(H_{a,0}),$$
so that the point $(\phi(H_{a,0}), \gamma(H_{a,0}), \tau(H_{a,0}))$ lies on the edge $P_4P_5$ of the polyhedron $\Omega$ for any $a\in [0, \tfrac12]$, and
$$  \tfrac23\phi(H_{a,\frac14})+\tfrac13 = \tfrac23(-3a^2+3a+\tfrac14) + \tfrac13 = -2a^2+2a+\tfrac12 = \tfrac12(1-2a)^2-4a^2+4a= \tau(H_{a,\frac14}),$$
so that the point $(\phi(H_{a,\frac14}), \gamma(H_{a,\frac14}), \tau(H_{a,\frac14}))$ lies on the edge $P_4P_6$ of the polyhedron $\Omega$ for any $a\in [0, \tfrac12]$. Similarly as in the proof of Lemma~\ref{lem:F6} we show that for any point $T$ on the face $F_7$ there exist $a\in [0, \tfrac12]$ and $b\in [0, \tfrac14]$ such that $(\phi(H_{a,b}), \gamma(H_{a,b}), \tau(H_{a,b})) = T$.
\end{example}

\begin{figure}[ht]
\centering
\includegraphics{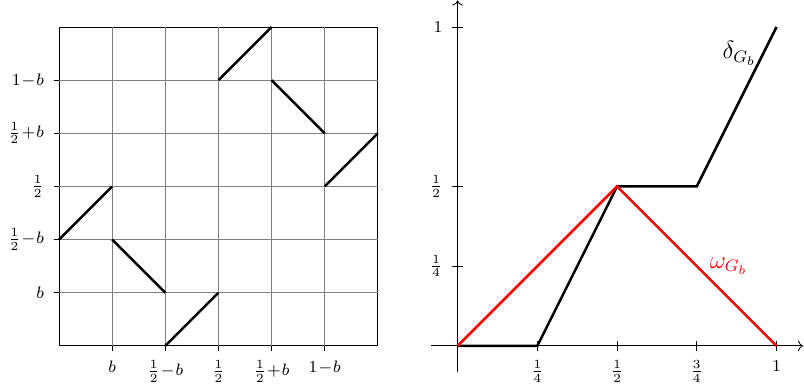}
    \caption{The mass distribution of copula $G_b$ (left) and the graphs of functions $\delta_{G_b}$ and $\omega_{G_b}$ (right) from Example \ref{ex:F7}.}
    \label{fig:ex F7}
\end{figure}

\begin{example}\label{ex:F1}
Consider the copula $K_{a,b} = H_{a,b}^{\sigma_2}$, where $H_{a,b}$ is the copula from Example~\ref{ex:F7}. 
Similarly as in the proof of Lemma~\ref{lem:F2}, we can prove that the set of points 
$$\big\{(\phi(K_{a,b}), \gamma(K_{a,b}), \tau(K_{a,b})) \mid (a,b) \in [0, \tfrac12]\times[0, \tfrac14]\big\}$$ 
is equal to the face $F_1 = P_1 P_2 P_3$ of polyhedron $\Omega$.
\end{example}

\begin{example}\label{ex:F5}
Let $a, b \in [0, \tfrac14]$ with $a \le b$ and let $L_{a,b}$ be a shuffle of $M$
\begin{align*}
    L_{a,b} = M\big(&10, (a, b, \tfrac12 - b, \tfrac12 - b + a, \tfrac12, \tfrac12 + b - a, \tfrac12 + b, 1-b, 1-a), \\
    &(7,5,8,10,2,9,1,3,6,4), (-1, 1, 1, -1, 1, 1, -1, 1, 1, -1)\big).
\end{align*}
The copula $L_{a,b}$ is defined by the measure preserving bijection 
$$ h_{L_{a,b}}(u) =  \begin{cases}
            \tfrac12 + b - u; &  0 \le u \le a, \\
            u + \tfrac12 - b; &  a < u \le b, \\
            u + \tfrac12; &  b < u \le \tfrac12 - b, \\
            \tfrac32 - b - u; &  \tfrac12 - b < u \le \tfrac12 - b + a, \\
            u - \tfrac12 + b; &  \tfrac12 - b + a < u \le \tfrac12, \\
            u + \tfrac12 - b; &  \tfrac12 < u \le \tfrac12 + b - a, \\
            \tfrac12 + b - u; &  \tfrac12 + b - a < u \le \tfrac12 + b, \\
            u - \tfrac12; &  \tfrac12 + b < u \le 1-b, \\
            u - \tfrac12 + b; &  1-b < u \le 1-a, \\
            \tfrac32 - b - u; &  1-a < u \le 1.
            \end{cases}
$$
We have 
$$ L_{a,b}(u,h_{L_{a,b}}(u)) =  \begin{cases}
            0; &  0 \le u \le a, \\
            u - a; &  a < u \le b, \\
            u; &  b < u \le \tfrac12 - b, \\
            \tfrac12 - b; &  \tfrac12 - b < u \le \tfrac12 - b + a, \\
            u - \tfrac12 + b - a; &  \tfrac12 - b + a < u \le \tfrac12, \\
            u - a; &  \tfrac12 < u \le \tfrac12 + b - a, \\
            0; &  \tfrac12 + b - a < u \le \tfrac12 + b, \\
            u - \tfrac12; &  \tfrac12 + b < u \le 1-b, \\
            u - \tfrac12 + b - a; &  1-b < u \le 1-a, \\
            \tfrac12 - b; &  1-a < u \le 1.
            \end{cases}
$$
Using equation \eqref{eq:tau-h}, we obtain
$$  \tau(L_{a,b}) = 16(b-a)^2 - 8b^2. $$
Furthermore, we have
$$ \delta_{L_{a,b}}(u) =  \begin{cases}
            0; &  0 \le u \le \tfrac12-b+a, \\
            2u - 1 + 2b - 2a; &  \tfrac12-b+a < u \le \tfrac12, \\
            2b - 2a; &  \tfrac12 < u \le \tfrac12+b-a, \\
            2u-1; &  \tfrac12+b-a < u \le 1, 
            \end{cases}
$$
and
$$ \omega_{L_{a,b}}(u) =  \begin{cases}
            u; &  0 \le u \le \tfrac14, \\
            \tfrac12 - u; &  \tfrac14 < u \le \tfrac12 - b + a, \\     
            u + 2b - 2a - \tfrac12; &  \tfrac12 - b + a < u \le \tfrac12, \\
            2b - 2a + \tfrac12 - u; &  \tfrac12 < u \le \tfrac12 + b - a, \\
            u - \tfrac12; &  \tfrac12 + b - a < u \le \tfrac34, \\
            1-u; &  \tfrac34 < u \le 1, 
            \end{cases}
$$
Using equations \eqref{phi}  and \eqref{gamma}, we obtain
$$  \phi(L_{a,b}) = 12(b-a)^2-\tfrac12, \quad \text{and} \quad \gamma(L_{a,b}) = 16(b-a)^2-\tfrac12. $$
The mass distribution of copula $L_{a,b}$ and the graphs of functions $\delta_{L_{a,b}}$ and $\omega_{L_{a,b}}$ are shown in Figure~\ref{fig:ex F5}.

Note that 
$$\tfrac43\phi(L_{a,b})+\tfrac16 = \tfrac43\big(12(b-a)^2-\tfrac12\big)+\tfrac16  
  = 16(b-a)^2-\tfrac12 = \gamma(L_{a,b}),$$
so the point $(\phi(L_{a,b}), \gamma(L_{a,b}), \tau(L_{a,b}))$ lies on the face $F_5$ for any $a, b\in [0, \tfrac14]$ with $a \le b$.
Furthermore, $L_{0,b} = C_b = A_{0,b}$,
so the point $(\phi(L_{0,b}), \gamma(L_{0,b}), \tau(L_{0,b}))$ lies on the edge $P_3P_6$ of the polyhedron $\Omega$ for any $b\in [0, \tfrac14]$, $\phi(L_{b,b}) = \gamma(L_{b,b}) = -\tfrac12$, so the point $(\phi(L_{b,b}), \gamma(L_{b,b}), \tau(L_{b,b}))$ lies on the edge $P_2P_3$ of the polyhedron $\Omega$ for any $b\in [0, \tfrac14]$, and
$$  \gamma(L_{a,\frac14}) = 16a^2-8a+\tfrac12 = \tau(L_{a,\frac14}),$$
so that the point $(\phi(L_{a,\frac14}), \gamma(L_{a,\frac14}), \tau(L_{a,\frac14}))$ lies on the line segment $P_2P_6$. For any fixed $b\in [0, \tfrac14]$ the set of points 
$$\big\{(\phi(L_{a,b}), \gamma(L_{a,b}), \tau(L_{a,b})) \mid a\in [0, b]\big\}$$ 
is a curve lying on the face $F_5$ and connecting a point on the edge $P_3P_6$ and a point on the edge $P_2P_3$. Since $\phi(L_{a,b}), \gamma(L_{a,b}), \tau(L_{a,b})$ are continuous as functions of $a,b$, this means that for any point $T$ on the triangle $P_2 P_3 P_6$ there exist $a, b \in [0, \tfrac14]$ and with $a \le b$ such that $(\phi(L_{a,b}), \gamma(L_{a,b}), \tau(L_{a,b})) = T$.

Finally, let $M_{a,b} = L_{a,b}^{\sigma_2}$. By Lemma~\ref{lem:linear_map_A} the linear mapping $\mathcal{A}$ maps triangle $P_2 P_3 P_6$ onto triangle $P_6 P_5 P_2$. Hence, by equation~\eqref{eq:mappingA}, for any point $T$ on the triangle $P_2 P_5 P_6$ there exist $a, b \in [0, \tfrac14]$ with $a \le b$ such that $(\phi(M_{a,b}), \gamma(M_{a,b}), \tau(M_{a,b})) = T$.
\end{example}

\begin{figure}[ht]
\centering
\includegraphics{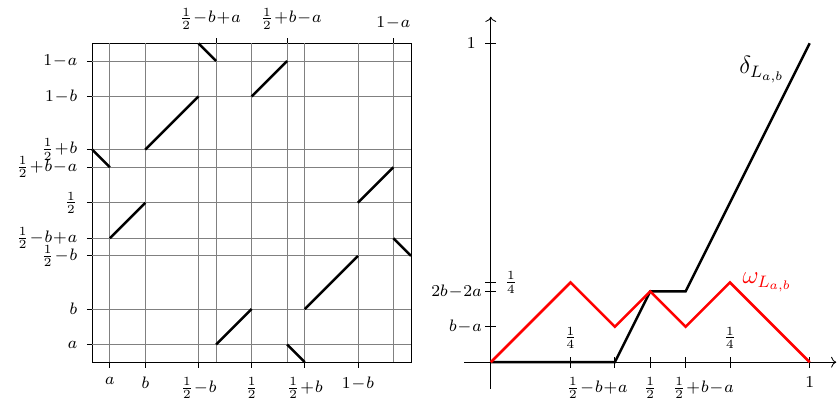}
    \caption{The mass distribution of copula $L_{a,b}$ (left) and the graphs of functions $\delta_{L_{a,b}}$ and $\omega_{L_{a,b}}$ (right) from Example \ref{ex:F5}.}
    \label{fig:ex F5}
\end{figure}

\section{Concluding remarks} \label{sec:concluding}

In this paper, we described the exact region $\Omega_{\phi, \gamma, \tau}$ determined by Spearman's footrule, Gini's gamma and Kendall's tau. It is a polyhedron $\Omega$ depicted in Figure~\ref{fig:polyhedron}, with $6$ vertices, $11$ edges and $7$ faces. We proved that every point on the boundary of $\Omega$ can be realized using a shuffle of $M$. 

The polyhedron $\Omega_{\phi, \gamma, \tau} $ has volume $\mathrm{vol}({\Omega_{\phi, \gamma, \tau} })=\frac{3}{16}=0.1875$, which represents $\frac{1}{32}=0.03125$ of the volume of the $(\phi,\gamma,\tau)$-box $[-\tfrac12,1] \times [-1,1] \times [-1,1]$.
In \cite{KoBuMo26}, the authors computed the volume of $\Omega_{\beta,\phi, \gamma}$, i.e., $\text{vol}(\Omega_{\beta,\phi, \gamma})=\frac{19}{40}=0.475$, which represents $\frac{19}{240} \approx 0.07916$ of the volume of the $(\beta,\phi,\gamma)$-box $[-1,1] \times[-\tfrac12,1] \times [-1,1]$.

Proposition \ref{pro:old-results} shows that for any copula $C$ with given value $\phi(C)$, the spread of $\tau(C)$, i.e., the difference between the maximal and the minimal possible value of $\tau(C)$, is at most $1$. On average, the spread of $\tau(C)$ given $\phi(C)$ is
$$\frac{\text{area}(\Omega_{\phi,\tau})}{\text{length}([-\frac12,1])}=\tfrac12,$$
since $\text{area}(\Omega_{\phi,\tau})=\frac34$, see \cite{KoBuSt}.
Similarly, given the value of $\gamma(C)$, the spread of $\tau(C)$ is at most $\frac23$. On average, the spread of $\tau(C)$ given $\gamma(C)$ is again
$$\frac{\text{area}(\Omega_{\gamma,\tau})}{\text{length}([-1,1])}=\tfrac12,$$
since $\text{area}(\Omega_{\gamma,\tau})=1$, see \cite{KoBuSt}.
If both $\phi(C)$ and $\gamma(C)$ are given, the maximal spread of $\tau(C)$ is still $\frac23$. Indeed, if $\phi(C)=\gamma(C)=0$, then $\tau(C)\in [-\frac13,\frac13]$ by Lemma \ref{lem:bounds}. The average spread of $\tau(C)$ given $\phi(C)$ and $\gamma(C)$ is
$$\frac{\mathrm{vol}({\Omega_{\phi, \gamma, \tau} })}{\text{area}(\Omega_{\phi,\gamma})}= \tfrac13,$$
since $\text{area}(\Omega_{\phi,\gamma})=\frac{9}{16}$, see \cite{KoBuMo}.
Using results from \cite{KoBuMo26}, the average spread of $\beta(C)$ given $\phi(C)$ and $\gamma(C)$ is
$$\frac{\mathrm{vol}({\Omega_{\beta,\phi, \gamma} })}{\text{area}(\Omega_{\phi,\gamma})}= \tfrac{38}{45}\approx 0.8444.$$
We conclude that the value of $\tau(C)$ is more closely related to the values of $\phi(C)$ and $\gamma(C)$ than the value of $\beta(C)$ is.

\section*{Acknowledgments}

The authors acknowledge financial support from the ARIS (Slovenian Research and Innovation Agency, research core funding No. P1-0222, research core funding No. P1-0448, and project J1-50002). The work of Petra Lazić was also supported by Croatian Science Foundation grant No. 2277.

\bibliographystyle{amsplain}
\bibliography{biblio}

\end{document}